\documentclass[12pt]{article}
\usepackage{amsmath,amsthm,amsfonts,amssymb}
\usepackage{graphicx}
\usepackage[utf8]{inputenc}
\usepackage{times}
\usepackage[T1]{fontenc}
\def\C{{\mathbb C}}
\def\D{{\mathbb D}}

\def\R{{\mathbb R}}
\def\S{{\mathbb S}}

\def\Z{{\mathbb Z}}
\def\dvf#1{{\underrightarrow{#1}}} 

\def\wt{\widetilde}
\def\wh{\widehat}
\def\ol{\overline}

\def\ur#1{\upshape{#1}}
\def\itemup#1{\item[\textup{#1}]}
\def\ch{\check}
\def\map{\colon}
\def\del{\partial}
\def\eps{\varepsilon}
\def\rst#1{{\upharpoonright}_{\!#1}}

\def\with{{\mathrel{:}}}
\def\id{\mathrm{id}}
\def\hook{\mathbin{\lrcorner}}
\def\Step#1){\noindent\hbox to \parindent{#1)\hfill}}
\theoremstyle{plain}
\newtheorem{theorem}{Theorem}
\newtheorem{corollary}[theorem]{Corollary}
\newtheorem{proposition}[theorem]{Proposition}
\newtheorem{lemma}[theorem]{Lemma}

\theoremstyle{definition}
\newtheorem{definition}[theorem]{Definition}

\newtheorem{example}[theorem]{Example}

\theoremstyle{remark}
\newtheorem*{acknowledgments}{Acknowledgments}
\newtheorem*{remark*}{Remark}
\newtheorem{remark}[theorem]{Remark}

\def\DD{{\mathcal{D}}}
\def\DDc{{\mathcal{D}_\del}}
\DeclareMathOperator{\MT}{MT}
\DeclareMathOperator{\OB}{OB}
\DeclareMathOperator{\Ker}{Ker}
\DeclareMathOperator{\Int}{Int}
\DeclareMathOperator{\MCG}{MCG}

\begin{document}

\title{IDEAL LIOUVILLE DOMAINS\\ A cool gadget}

\author{Emmanuel \textsc{Giroux}\thanks%
{\emph{Centre National de la Recherche Scientifique} (UMI 3457 CRM) and \emph
{Universit\'e de Montr\'eal}. Partially supported by the ANR grant \emph{MICROLOCAL}
(ANR-15CE40-0006).}}

\date{Montreal --- August 2017}

\maketitle

The purpose of these notes is to describe a convenient packaging for those
objects nowadays called Liouville domains but which have been studied formerly
under various names such as ``symplectic manifolds with restricted contact type
boundaries'' or ``complete convex symplectic manifolds with conical/cylindrical
ends'' \cite{We, EG, Mc, Ge, La, CFH, Vi, Se, CE}. In this text, the term \emph
{domain} systematically refers to a compact manifold with boundary.

\begin{definition}[Liouville Domains] \label{d:ld}
A \emph{Liouville domain} is a domain $F$ endowed with a \emph{Liouville form},
that is, a $1$-form $\lambda$ satisfying the following two axioms:
\begin{itemize}
\item
$\omega := d\lambda$ is a symplectic form on $F$.
\item
$\lambda$ induces a contact form on $K := \del F$ orienting $K$ as the boundary
of $(F,\omega)$.
\end{itemize}
\end{definition}

Liouville domains are ubiquitous in symplectic geometry. Obvious examples are 
starshaped tubes about the zero-section in cotangent spaces of closed manifolds.
More generally, Stein domains are fundamental examples. In addition, it follows
from Donaldson's work \cite{Do} that every closed integral symplectic manifold
can be obtained from a Liouville domain (whose Reeb flow on the boundary defines
a free circle action) by a ``symplectic reduction'' crashing the boundary to its
quotient by the circle action. Similarly, every closed contact manifold can be
constructed from a Liouville domain as the ``relative mapping torus'' of some
symplectic self-diffeomorphism fixing the boundary points \cite{Gi1}.

On each Liouville domain, there is a wide (open) choice of Liouville forms. In
particular, any Liouville form $\lambda$ has many multiples $e^f \lambda$ which
are Liouville forms as well: The pertinent condition on the function $f$ is that
$\dvf\lambda \cdot f > -1$, the \emph{Liouville vector field} $\dvf\lambda$
being given by $\dvf\lambda \hook d\lambda = \lambda$. By the second axiom of
Definition \ref{d:ld}, $\dvf\lambda$ points transversely outwards along $\del F$
(because $\dvf\lambda \hook (d\lambda)^n = n\,\lambda \wedge (d\lambda)^{n-1}$),
so the inequality $\dvf\lambda > -1$ admits a big convex set of solutions~$f$.
However, these rescaled Liouville forms share important geometric features: They
determine the same contact structure $\xi$ (with all possible contact forms) on
$\del F$, the singular foliations spanned by their Liouville fields coincide,
and the symplectic structures they define on $F$ are the same up to completion
(and sliding in the direction of the Liouville fields). The \emph{completion} of
a Liouville domain $(F,\lambda)$ is an open manifold $\wh F \supset F$ equipped
with a $1$-form $\wh\lambda$ such that:
\begin{itemize}
\item
$\wh\omega := d\wh\lambda$ is a symplectic form on $\wh F$,
\item
$\wh\lambda \rst F$ equals $\lambda$, and
\item
$\dvf\lambda$ is a complete vector field whose flow induces a diffeomorphism
from $\del F \times \R_{\ge0}$ to $\wh F - \Int F$.
\end{itemize}
It is easy to check that such a completion always exists and is unique up to
symplectomorphism. More precisely, between any two completions of the same domain
$(F,\lambda)$, there is a unique diffeomorphism which is the identity on $F$ and
conjugates the extended Liouville forms. 

The completion $(\wh F, \wh\lambda)$ offers an alternative description of the
contact manifold $(\del F, \xi)$ (with no preferred contact form) as the orbit
space of $\dvf\lambda$ at infinity. In \cite{EKP}, this orbit space is called
the \emph{ideal contact boundary} of $(\wh F, \wh\lambda)$. A natural question
then arises: Does this ideal contact boundary really depend on the form $\wh
\lambda$ (with well-behaved dual vector field) or only on the symplectic form
$\wh\omega := d\wh\lambda$? This question was studied by S.~Courte in his thesis
and he exhibited examples of completions with isomorphic symplectic structures
but non-diffeomorphic ideal contact boundaries \cite{Co1, Co2}. The observation
leading to the concept of ideal Liouville domains is that this ambiguity about
the ideal contact boundary can be lifted by fixing a smooth compactification of
the completion ``taming'' the symplectic structure. Ideal Liouville domains are
domains with a symplectic structure in the interior (subject to some ``tameness 
condition'' near the boundary) which uniquely determines a contact structure on
the boundary. They have an affine space of (ideal) Liouville forms (and none of
them being part of the data) whose dual vector fields are complete and hit the
boundary transversely. They enjoy the stability that is expected from symplectic
objects (for the suitable topology) and they can be manipulated (modified and
combined) by means of various operations, notably products (without corners),
handle attachments (in particular, boundary connect sums), and plumbings (along
proper Lagrangian disks with Legendrian boundaries). They are also very useful
to understand and describe the relationships between contact structures and open
books, and they were actually introduced in this context. The second half of the
paper is devoted to the contact aspects. After discussing subtleties related to
the monodromy of open books, we define and study Liouville open books. We prove
that every open book supporting a contact structure is a Liouville open book and
that every Liouville open book supports an essentially unique contact structure.

\begin{acknowledgments}
I am indebted to Robert Roussarie for his nice observation reproduced in Remark
\ref{r:smooth}. I also wish to thank Sylvain Courte for his helpful comments on
the preliminary version of this text he used to write his thesis \cite{Co2}.
Finally, I am grateful to Patrick Massot and Klaus Niederkr\"uger for adopting and
advertising the notion of ideal Liouville domains, and for pushing me to write
these notes.
\end{acknowledgments}

\subsection{Ideal Liouville domains in their own}

\begin{definition}[Ideal Liouville Domains] \label{d:ild}
An \emph{ideal Liouville domain} $(F,\omega)$ is a domain $F$ endowed with an
ideal Liouville structure $\omega$. This \emph{ideal Liouville structure} is an
exact symplectic form on $\Int F$ admitting a primitive $\lambda$ such that: For
some (and then any) function $u \map F \to \R_{\ge0}$ with regular level set
$\del F = \{u=0\}$, the product $u \lambda$ extends to a smooth $1$-form on $F$
which induces a contact form on $\del F$.

A $1$-form $\lambda$ as above is called an \emph{ideal Liouville form}. Its dual
vector field $\dvf\lambda$ is an \emph{ideal Liouville field}.
\end{definition} 

Liouville forms in ideal Liouville domains are analogous to contact forms on
contact manifolds: They exist, and it is sometimes useful to choose one, but
this choice is most often unimportant. In this analogy, ideal Liouville fields
correspond to Reeb fields though their dynamics are very different: The latter
are Hamiltonian while the former expand the symplectic form ---~and hence the
volume~--- exponentially. Speaking of contact manifolds, one of the striking
features of ideal Liouville domains is:

\begin{proposition}[The Boundary Contact Structure] \label{p:contact}
Let $(F,\omega)$ be an ideal Liouville domain. Then the boundary $K := \del F$
has a positive contact structure $\xi$, uniquely determined by $\omega$, which
is left invariant by any diffeomorphism of $F$ preserving $\omega$. Moreover,
the positive equations of $\xi$ are in one-to-one correspondence with the
negative sections of the conormal bundle of~$K$. 
\end{proposition}

As a consequence of the latter claim, every symplectic diffeomorphism $\phi$ of 
$(F,\omega)$ (meaning a diffeomorphism of $F$ whose restriction to the interior
preserves $\omega$) which is relative to the boundary (in the sense that its
restriction to $\del F$ is the identity) is actually tangent to the identity at
all points of $\del F$. Indeed, $\phi$ preserves the equations of the boundary
contact structure and hence acts trivially on the conormal bundle of $\del F$.  

\begin{proof}
Let $\lambda$ be a Liouville form and $u \map F \to \R_{\ge0}$ a function with
regular level set $K = \{u=0\}$. By assumption, $u \lambda$ extends to a smooth
$1$-form $\beta$ on $F$ which induces a contact form on $K$. Write
$$ \omega = d\lambda = d(\beta/u)
 = u^{-2} (u\, d\beta + \beta \wedge du). $$
This formula demonstrates that $u^2\omega$ extends to a smooth $2$-form $\gamma$
on $F$ which depends on $u$ only up to a conformal factor. Now, along $K$, the
form $\gamma$ has rank $2$, and its kernel is the contact structure $\xi$ on $K$
defined by $\beta$. It follows that $\xi$ is independent of the choice of
$\lambda$ and~$u$. Moreover, the identity $\gamma = \beta \wedge du$ shows that
$\beta \rst K$ is also independent of $\lambda$ and is uniquely (and pointwise
linearly) determined by $du$ viewed as a section of the conormal bundle of $K$. 
\end{proof}

Now recall that the symplectization of a contact manifold $(K,\xi)$ is the 
symplectic submanifold $SK$ of $T^*K$ consisting of non-zero covectors $\beta_p
\in T^*_pK$, $p \in K$, whose cooriented kernel is $\xi_p$ (contact structures 
are cooriented in this text). We denote by $\lambda_\xi$ the $1$-form induced on
$SK$ by the canonical Liouville form of $T^*K$. We also define the ``projective
completion'' of $SK$ as the quotient
$$ \ol{SK} := (SK \times \R_{\ge0}) \big/ \R_{>0} $$
where $\R_{>0}$ acts (freely, properly and) diagonally by multiplication. Thus,
$\ol{SK}$ is a smooth manifold with boundary obtained by attaching a copy of
$K = SK/\R_{>0}$ to $SK = (SK \times \R_{>0}) / \R_{>0}$.

\begin{proposition}[Ideal Liouville Fields] \label{p:lioufields}
Let $(F,\omega)$ be an ideal Liouville domain, $(K,\xi)$ its contact boundary,
and $\lambda$ an ideal Liouville form in $\Int F$.

\itemup{a)}
The Liouville field $\dvf\lambda$ is complete and the singular foliation spanned
by $\dvf\lambda$ extends to a foliation of $F$ which is non-singular along $K$
and transverse to $K$. We denote by $U$ the open collar neighborhood of $K$
consisting of all extended leaves reaching $K$.

\itemup{b)}
There exists a unique embedding $\iota = \iota_\lambda \map \ol{SK} \to F$ such
that $\iota \rst K = \id$ and $\iota^*\lambda = \lambda_\xi$; its image is the
open collar neighborhood $U$.
\end{proposition}

\begin{proof}
Let $u \map F \to \R_{\ge0}$ be a function with regular level set $K = \{u=0\}$
and $\beta$ the form extending $u \lambda$ over~$F$. For $n = \frac12 \dim F$,
\begin{multline*}
   \omega^n = \bigl(d (\beta / u)\bigr)^n
 = u^{-2n} (u\, d\beta + \beta \wedge du)^n \\
 = u^{-n-1} (u\, d\beta + n \beta \wedge du) \wedge (d\beta)^{n-1}  
 = u^{-n-1} \mu 
\end{multline*}
where $\mu := (u\, d\beta + n\beta \wedge du) \wedge (d\beta)^{n-1}$ is a volume
form on~$F$.

Let $\nu$ denote the vector field on $F$ given by $\nu \hook \mu = n\beta \wedge
(d\beta)^{n-1}$. Since $\beta$ induces a positive contact form on the boundary,
$\nu$ is non-singular along $K$ and points transversely outwards (specifically,
$\nu \cdot u = -1$ by the very definition of things). On the other hand, in the
interior of $F$,
$$ n \beta \wedge (d\beta)^{n-1} 
 = n u^n \lambda \wedge (d\lambda)^{n-1}
 = u^n \dvf\lambda \hook \omega^n
 = u^{-1}\, (\dvf\lambda \hook \mu). $$
Comparing this relation with the definition of $\nu$, we get $\dvf\lambda = u
\nu$. Since $u$ vanishes along $K$, the vector field $\dvf\lambda$ is complete
and the foliation it defines extends to the foliation spanned by $\nu$. This
proves Part a).

As for Part b), first note that if the embedding $\iota$ exists then it maps the
standard Liouville field $\dvf{\lambda_\xi}$ of $SK$ to $\dvf\lambda$, and so its
image has to be $U$. Now observe that the holonomy of the foliation spanned by
$\nu$ yields a projection $U \to K$ and, for any point $p \in U-K$ projecting to
$q \in K$, identifies $\lambda_p \in T_p^*U$ with a covector in $T_q^*K$ whose
cooriented kernel equals $\xi_q$ (just because the holonomy preserves the kernel
of $\beta = u\lambda$). Thus, we have a smooth map $U \to \ol{SK}$ which is the
identity on $K$. The expansion properties of the flow of $\dvf\lambda$ imply
that this map is a diffeomorphism, and we define $\iota$ to be the inverse map.
The relation $\iota^*\lambda = \lambda_\xi$ follows from the very definition of
$\lambda_\xi$, and $\iota$ is unique because the identity of $K$ lifts to a
unique diffeomorphism of $SK$ preserving $\lambda_\xi$.
\end{proof}

\begin{corollary}[Ideal Liouville Forms] \label{c:liouforms} 
On any ideal Liouville domain $(F,\omega)$, ideal Liouville forms constitute an
affine space. Given a function $u \map F \to \R_{\ge0}$ with regular level set
$\del F = \{u=0\}$, the underlying vector space can be described as consisting
of all closed $1$-forms $\kappa$ on $\Int F$ satisfying the following equivalent
conditions:
\begin{itemize}
\itemup{(i)}
The form $u\kappa$ extends to a smooth form on $F$.
\itemup{(ii)}
The vector field $\dvf\kappa/u$ extends to a smooth vector field on $F$
\ur(which is automatically tangent to $K := \del F$\ur).
\itemup{(iii)}
There exists a function $f \map F \to \R$ such that $\kappa - d (f \log u)$ is
the restriction of a closed $1$-form on $F$.
\end{itemize}
\end{corollary}

As a result, a Lagrangian submanifold $L \subset \Int F$ is exact for some ideal
Liouville form if and only if its Liouville class (with respect to an arbitrary
given ideal Liouville form) lies in the image of the natural map $H^1(F,\R) \to
H^1(L,\R)$.

\begin{proof}
The only (maybe) non-trivial claim which is not a straightforward consequence of 
Propositions \ref{p:contact} and \ref{p:lioufields} is that (i) implies (iii).
So assume that $u \kappa$ extends to a smooth form $\gamma$ on $F$. In $\Int F$,
$$ 0 = d\kappa = d(\gamma/u) = u^{-2} (u\, d\gamma + \gamma \wedge du). $$
By continuity, $u\, d\gamma + \gamma \wedge du$ is identically zero on $F$, and
hence $\gamma \wedge du = 0$ along $K = \{u=0\}$. Thus, there exists a function
$f \map K \to \R$ such that $\gamma = f\, du$ along $K$. Extend $f$ (keeping its
name) to a function on $F$ and observe that the form
$$ \gamma - u\, d (f \log u) = \gamma - u \log u \, df - f\, du $$
extends to a $1$-form $\gamma'$ on $F$ which vanishes identically along $K$. It
follows that $\gamma' = u \kappa'$ where $\kappa'$ is a closed $1$-form on $F$.
\end{proof}

Another corollary is the following avatar of a standard lemma (see Lemma 1.1 and
the subsequent remark in \cite{BEE}): 

\begin{corollary}[Exact Isotopies] \label{c:exact}
Let $(F,\omega)$ be an ideal Liouville domain and $\lambda_t$ $(t \in [0,1])$ a
path of ideal Liouville forms in $\Int F$. Then there is a symplectic isotopy 
$\psi_t$ $(t \in [0,1])$ of $F$, relative to the boundary, such that $\psi_0 =
\id$ and, for every $t \in [0,1]$, the form $\psi_t^*\lambda_t - \lambda_0$ is
the differential of a function with compact support in $\Int F$. 
\end{corollary}

Here the path $\lambda_t$ is assumed to be smooth in the sense that $\lambda_t =
\beta_t/u$ where $\beta_t$ $(t \in [0,1])$ is a smooth path of $1$-forms on $F$,
\emph{ie.} a smooth $1$-form on $[0,1] \times F$ whose contraction with $\del_t$
is zero ($u$, as usual, is a non-negative function on $F$ with regular level set
$K := \del F = \{u=0\}$). 

\begin{proof}
For $t \in [0,1]$, let $\iota_t \map \ol{SK} \to F$ be the unique embedding such
that $\iota_t \rst K = \id$ and $\iota_t^*\lambda_t = \lambda_\xi$, where $\xi$
is the boundary contact structure (cf.~Proposition \ref{p:lioufields}). Setting
$U_t := \iota_t (\ol{SK})$, we have an isotopy of embeddings
$$ \psi_t^0 := \iota_t \circ \iota_0^{-1} \map U := U_0 \to F $$
with the following properties:
\begin{itemize}
\item
$\psi_0^0 = \id$ and $\psi_t^0 \rst K = \id$ for all $t \in [0,1]$,
\item
$(\psi_t^0)^*\lambda_t = \lambda_0$ on $U-K$ for all $t \in [0,1]$.
\end{itemize}
Therefore, the time-dependent vector field $\eta_t^0$ on $U_t$ generating the
isotopy $\psi_t^0$ satisfies, for all $t \in [0,1]$,
\begin{itemize}
\item
$\eta_t^0 = 0$ along $K$, and
\item
$(\eta_t^0 \cdot \lambda_t) + \dot\lambda_t = 0$ on $U_t - K$.
\end{itemize}
Let $f_t^0 := \lambda_t(\eta_t^0)$ and denote by $\eta_t^1$ the time-dependent
locally Hamiltonian vector field on $\Int F$ given by $\eta_t^1 \hook \omega =
-\dot\lambda_t$. In $U_t-K$,
$$ (\eta_t^1 - \eta_t^0) \hook \omega = df_t^0 . $$ 
Now take a time-dependent function $f_t$ on $F$ equal to $f_t^0$ near $K$, and
consider the locally Hamiltonian vector field $\eta_t$ on $\Int F$ such that
$(\eta_t^1 - \eta_t) \hook \omega = df_t$. Since $\eta_t = \eta_t^0$ close to
the boundary, $\eta_t$ extends smoothly to a vector field on $F$ which vanishes
identically along $K$. On the other hand, in $\Int F$,
$$ (\eta_t \cdot \lambda_t) + \dot\lambda_t  
 = d \bigl(\lambda_t(\eta_t) - f_t\bigr), $$ 
and the function $\lambda_t(\eta_t) - f_t$ is zero on the neighborhood of $K$
where $f_t = f_t^0$ (and $\eta_t = \eta_t^0$). The desired isotopy $\psi_t$ is
obtained by integrating the vector field $\eta_t$.
\end{proof}

Ideal Liouville domains are stable in the following sense:

\begin{lemma}[Stability] \label{l:moser} \label{l:stability}
Let $F$ be a domain and $(\omega_t)$ $(t \in [0,1])$ a path of ideal Liouville
structures on $F$. Then there exists an isotopy $\phi_t$ $(t \in [0,1])$ of $F$
such that $\phi_0=\id$ and $\phi_t^*\omega_t = \omega_0$ for all $t \in [0,1]$.
Moreover, we can choose this isotopy relative to $K = \del F$ if ---~and clearly
only if~--- all forms $\omega_t$ induce the same boundary contact structure. 
\end{lemma}

Here again, the required smoothness of the path $\omega_t$ is that there is a
smooth path $\beta_t$ of $1$-forms on $F$ such that $\omega_t = d(\beta_t/u)$.

\begin{proof}[Sketch of proof]
Due to the smoothness of the path $\omega_t$, the induced contact structure on
$K$ vary smoothly with $t$. Then, by Gray's stability Theorem (and the obvious
fact that any isotopy of $K$ extends to an isotopy of $F$), it suffices to treat
the case when all forms $\omega_t$ induce the same boundary contact structure
$\xi$. Using Proposition \ref{p:lioufields}, we can further arrange that the
forms $\omega_t$ coincide near $K$ and, more specifically, have smoothly varying
ideal Liouville forms $\lambda_t$ which all agree in a neighborhood of $K$. Then
we conclude with Moser's standard argument. 
\end{proof}

The next proposition is another expected and straightforward result relating the
symplectic geometry of (the interior of) an ideal Liouville domain $(F,\omega)$
with the contact geometry of its boundary $(K,\xi)$. The notations are as
follows:
\begin{itemize}
\item
$\DD (F, \omega)$ is the group of diffeomorphisms of $F$ preserving $\omega$,
\item
$\DDc (F,\omega) \subset \DD (F,\omega)$ is the subgroup of diffeomorphisms
fixing $K := \del F$ pointwise, and
\item 
$\DD (K,\xi)$ is the group of diffeomorphisms of $K$ preserving $\xi$.
\end{itemize}

\begin{proposition}[Relations between Automorphism Groups] \label{p:fibration}
Let $(F,\omega)$ be an ideal Liouville domain with contact boundary $(K,\xi)$.
The restriction homomorphism 
$$ \DD(F,\omega) \to \DD(K,\xi) $$
is a Serre fibration, with associated long exact sequence of homotopy groups 
$$ \ldots \pi_k \DDc (F,\omega) \to \pi_k \DD (F,\omega) \to 
   \pi_k \DD (K,\xi) \to \pi_{k-1} \DDc (F,\omega) \ldots\,. $$
\end{proposition}

The homomorphism $\pi_1 \DD(K,\xi) \to \pi_0 \DDc(F,\omega)$ can be used to
define natural semigroups in the symplectic mapping class group $\MCG (F,\omega)
:= \pi_0 \DDc(F,\omega)$: An element in there is positive (resp.\ non-negative)
if it is the image of a positive (resp.\ non-negative) loop in $\DD (K,\xi)$.
When $(K,\xi)$ is a ``contact circle bundle'' (meaning that some Reeb flow
generates a free circle action), the image of the corresponding loop is the
mapping class of a ``fibered Dehn twist''. 

\begin{proof}[Sketch of proof]
We merely explain how to lift paths. Let $\phi_0 \in \DD(F,\omega)$ and take a
path $\ch\phi_t \in \DD(K,\xi)$ $(t \in [0,1])$ starting with $\ch\phi_0 =
\phi_0 \rst K$. The contact isotopy $\ch\phi_t$ lifts to a Hamiltonian isotopy
$S\ch\phi_t$ in the symplectization $SK$. Pick an arbitrary ideal Liouville form
$\lambda$ and identify $\ol{SK}$ with the collar neighborhood $U = \iota_\lambda
(\ol{SK})$ of Proposition \ref{p:lioufields}. The path $S\ch\phi_t$ can then be
viewed as a Hamiltonian isotopy of $U$ extending $\ch\phi_t$. We obtain the path
$\phi_t$ by cutting off the corresponding Hamiltonian functions away from $K$
inside $U$, and by integrating the new Hamiltonian functions with $\phi_0$ as
the initial condition.
\end{proof}

We now describe the two main examples of ideal Liouville domains.

\begin{example}[\textit{in vivo}: Convex Hypersurfaces in Contact Manifolds]
\label{x:ild1}
Let $(V,\xi)$ be a contact manifold and $S$ a hypersurface in $V$ which is \emph
{$\xi$-convex}, meaning that $S$ is transverse to some contact vector field
$\nu$. Consider the ``dividing set''
$$ \Gamma := \{p \in S \with \nu_p \in \xi_p\} \subset S. $$ 
Then the closure of every relatively compact connected component of $S - \Gamma$
is naturally an ideal Liouville domain (see \cite[I.3-C]{Gi1}).

To see this, pick an equation $\alpha$ of $\xi$, set $u := \alpha(\nu)$ and note
that $\Gamma$ is the zero-set of $u \rst S$. We claim that $u \rst S$ vanishes
transversely. Indeed, the identity $du \rst \xi = -(\nu \hook d\alpha) \rst \xi$
(drawn from the Cartan formula for the Lie derivative) implies that $du \rst{
\xi \cap TS } \ne 0$ along $\Gamma$, and that $\Gamma$ is actually a contact
submanifold of $(V,\xi)$. Moreover, $\nu$ restricted to the open set $\{u\ne0\}
\subset V$ is the Reeb vector field of the contact form $\alpha/u$. Since $\nu$
is transverse to $S$, the differential $d(\alpha/u)$ induces a symplectic form
on $S - \Gamma$.
\end{example}

\begin{example}[\textit{in vitro}: Ideal Completion of a Liouville Domain]
\label{x:ild2}
Let $(F,\lambda)$ be a Liouville domain in the sense of Definition~\ref{d:ld},
and let $u \map F \to \R_{\ge0}$ be a function with the following properties:
\begin{itemize}
\item
$u$ admits $K := \del F$ as its regular level set $\{u=0\}$,
\item
$\dvf\lambda \cdot \log u < 1$ at every point in $\Int F$.
\end{itemize}
Then a simple calculation (already resorted to in the introduction) shows that
$\omega := d(\lambda/u)$ is a symplectic form on $\Int F$, and so $(F,\omega)$
is an ideal Liouville domain. Moreover, since conditions 1 and 2 define a convex
cone of functions $u$, it follows from Lemma \ref{l:moser} that, up to isotopy
relative to the boundary, the geometry of $(F,\omega)$ is independent of~$u$.
Taking $u$ non-increasing along the orbits of $\dvf\lambda$ and equal to $1$
outside the collar neighborhood of $K$ associated with $\lambda$ (by Proposition
\ref{p:lioufields}), we see that $(\Int F, \omega)$ is symplectically isomorphic
to the completion of $(F,\lambda)$. For this reason, $(F,\omega)$ is called the
\emph{ideal completion} of $(F,\lambda)$. It can be alternatively obtained by
gluing $K$ to the usual completion $(\wh F, \wh\lambda)$ in exactly the same way
as $\ol{SK}$ was constructed from $SK$. 
\end{example}

To conclude this general discussion of ideal Liouville domains, here is the
product construction alluded to in the introduction:

\begin{proposition}[Product of Ideal Liouville Domains]
Let $(F_1,\omega_1)$ and $(F_2,\omega_2)$ be two ideal Liouville domains. Up to
isomorphism, there exists a unique ideal Liouville domain $(F,\omega)$ admitting
a diffeomorphism $\phi \map \Int F \to \Int (F_1 \times F_2)$ such that $\omega
= \phi^* (\omega_1 \oplus \omega_2)$ and, for any Liouville forms $\lambda_1$
and $\lambda_2$ on $F_1$ and $F_2$, respectively, $\phi^* (\lambda_1 \oplus
\lambda_2)$ is a Liouville form on $F$.
\end{proposition}

\begin{proof}
Clearly, $(\Int (F_1 \times F_2), \lambda_1 \oplus \lambda_2)$ is the (usual)
completion of some Liouville domain. The desired product is the ideal completion
of this domain. Uniqueness follows from the convexity of the sets of ideal
Liouville forms on $F_1$ and $F_2$.
\end{proof}

\begin{remark}[Generalizations] \label{r:gen}
I presented the notion of ideal Liouville domains in a talk at ETH (Zurich) in
November 2010 (for Eddi Zehnder's 70th birthday), and the first published paper
where they explicitly appear is \cite{MNW}. The concept was further generalized
in \cite{Co2} where Courte defined ideal Liouville cobordisms. A cobordism is an
oriented domain $F$ whose boundary components are given prescribed orientations;
$\del_+F$ (resp.~$\del_-F$) denotes the union of the boundary components endowed
with the boundary orientation (resp.\ with the reversed orientation). An \emph
{ideal Liouville cobordism} is a cobordism $F$ together with an exact symplectic
form $\omega$ on $\Int F$ which admits a primitive $\lambda$ such that:
\begin{itemize}
\item
For some/any function $u \map F \to \R_{\ge0}$ with regular level set $\del_+F
= \{u=0\}$, the product $u \lambda$ extends to a smooth $1$-form on $\Int F \cup
\del_+F$ which induces a contact form on $\del_+F$.
\item
For some/any function $u \map F \to \R_{\ge0}$ with regular level set $\del_-F
= \{u=0\}$, the quotient $\lambda / u$ extends to a smooth $1$-form on $\Int F
\cup \del_-F$ which induces a contact form on $\del_-F$.
\end{itemize}
Thus, an ideal Liouville domain $(F,\omega)$ is an ideal Liouville cobordism for
which $\del_-F$ is empty.

All the results discussed above readily extend to ideal Liouville cobordisms. In
particular, both $\del_-F$ and $\del_+F$ inherit canonical contact structures
which are positive for their prescribed orientations. In other words, $\del_-F$
is concave while $\del_+F$ is convex.  

Finally, the global exactness condition on the symplectic form can be relaxed
since exactness is needed only near the boundary. This leads to the definition
of \emph{ideal symplectic domains/cobordisms}. 
\end{remark}

\subsection{Ideal Liouville domains in contact geometry}

We will now explain how the notion of ideal Liouville domain can help in the
study of the relationships between contact structures and open books. We begin 
with a few basic definitions and constructions.

\medskip

An \emph{open book} in a closed manifold $V$ is a pair $(K,\theta)$, where:
\begin{itemize}
\item
$K \subset V$ is a submanifold of codimension $2$ with trivial normal bundle.
\item
$\theta \map V-K \to \S^1 = \R/2\pi\Z$ is a smooth locally trivial fibration 
which, in some neighborhood $\D^2 \times K$ of $K = \{0\} \times K$, is simply
the (pullback of the) angular coordinate in $\D^2 - \{0\}$.
\end{itemize}
The submanifold $K$ is called the \emph{binding} of the open book while the 
closures of the fibers of $\theta$ are the \emph{pages}. The binding and the
pages inherit coorientations from the canonical orientation of $\S^1$. Hence, if
$V$ is oriented, they are automatically oriented (and the binding is oriented as
the boundary of every page).

In practice, most often open books arise from (smooth) complex-valued maps. If a
map $h \map V \to \C$ vanishes transversely, with zero-set $K := \{h=0\}$, and
if the argument function $\theta := h/|h| \map V-K \to \S^1$ has no critical
points, then the pair $(K,\theta)$ is an open book. Obviously, every open book
$(K,\theta)$ can be obtained in this way, and the defining map $h$ is unique up
to multiplication by a positive function.
 
\medskip

An open book $(K,\theta)$ in a closed manifold $V$ is characterized by its 
monodromy, which is a diffeomorphism of the $0$-page $F := K \cup \{\theta=0\}$
relative to the boundary $K$ and defined only up to isotopy. More precisely,
consider the affine space of \emph{spinning vector fields}, namely vector fields
$\nu$ on $V$ satisfying the following properties:
\begin{itemize}
\item
$\nu = 0$ along $K$ and $\nu \cdot \theta = 2\pi$ in $V-K$;
\item
$\nu$ is \emph{weakly smooth} in the sense that it lifts to a smooth vector
field on the manifold with boundary obtained from $V$ by a real oriented blowup
along $K$ (see Remark \ref{r:smooth} for comments on this condition). 
\end{itemize}
For any such vector field $\nu$, the time~$1$ map of its flow, restricted to
$F$, is a diffeomorphism $\phi$ of $F$ relative to $K$. Moreover, as $\nu$ runs
over its affine space, $\phi$ sweeps out an entire mapping class in $\MCG(F) :=
\pi_0 \DDc(F)$ (cf.~Remark \ref{r:smooth}). This mapping class ---~and sometimes
also, by extension, any of its representatives~--- is the \emph{monodromy} of
the open book $(K,\theta)$. 

Conversely, given a domain $F$ with non-empty boundary and a diffeomorphism
$\phi$ of $F$ relative to $K := \del F$, one can construct a closed manifold
$\OB(F,\phi)$ endowed with an obvious open book whose $0$-page is parametrized
by $F$ and whose monodromy is represented by $\phi$. There are two steps in the
construction.

\Step1)
We consider the mapping torus of $\phi$, namely the quotient
$$ \MT(F,\phi) := (\R \times F) \big/ {\sim} \quad
   \text{where} \quad (t,p) \sim \bigl(t-1, \phi(p)\bigr). $$ 
This is a compact manifold (with boundary) which has an obvious fibration
$$ \wh\theta \map \MT(F,\phi) \to \S^1 = \R/2\pi\Z $$
coming from the projection $\R \times F \to \R$ multiplied by $2\pi$. All fibers
are diffeomorphic to $F$ and we use the projection $\R \times F \to \MT(F,\phi)$
restricted to $\{0\} \times F$ as a special parametrization of the $0$-fiber
$\{\wh\theta=0\}$ by $F$. We notice that, since $\phi$ induces the identity on
$K = \del F$, the boundary of $\MT(F,\phi)$ is canonically diffeomorphic to
$\S^1 \times K$, the restriction of $\wh\theta$ to $\del\MT(F,\phi)$ being given
by the projection $\S^1 \times K \to \S^1$.

An important point about the manifold $\MT(F,\phi)$ is that it depends only on
the mapping class of $\phi$ in the following sense: If $\phi_0$ and $\phi_1$ are
diffeomorphisms of $F$ relative to $K$ and representing the same mapping class
in $\MCG(F)$, then there is a diffeomorphism $\MT(F,\phi_0) \to \MT(F,\phi_1)$
which respects the fibrations over $\S^1$ and the special parametrizations of
the $0$-fibers.

\Step2) 
We construct the closed manifold $\OB(F,\phi)$ from $\MT(F,\phi)$ by collapsing
every circle $\S^1 \times \{.\} \subset \S^1 \times K = \del\MT(F,\phi)$ to a
point. Thus, $\OB(F,\phi)$ is the union of $\Int \MT(F,\phi)$ and $K = (\S^1
\times K) / \S^1$. We denote by 
$$ \theta \map \OB(F,\phi) - K = \Int \MT(F,\phi) \to \S^1 $$
the restriction of the fibration $\wh\theta$.

To see that $(K,\theta)$ is indeed an open book in $\OB(F,\phi)$, we need to
specify the smooth structure near $K$. In short, we blow down $\del\MT(F,\phi)$,
the points of $\del\MT(F,\phi) = \S^1 \times K$ corresponding to oriented lines
in the (trivial) normal bundle of $K$ in $\OB(F,\phi)$. Concretely, we fix a
collar neighborhood $\wh N$ of $\del\MT(F,\phi)$ whose fibers are intervals
contained in the fibers of $\wh\theta$ and we declare that, for every $p \in K$,
the union of all intervals ending on $\S^1 \times \{p\} \subset \del\MT(F,\phi)$
projects to a smooth disk $D_p$ in $\OB(F,\phi)$ transverse to $K$ at $p$. More
specifically, we choose a function $\wh r \map \MT(F,\phi) \to \R_{\ge0}$ with
regular level set $\del\MT(F,\phi) = \{\wh r = 0\}$, and we take the induced
function $r$ on $\OB(F,\phi)$, together with $\theta$, as polar coordinates near
$p$ on the disk $D_p$.

It is not hard to check that a different choice of collar neighborhood $\wh N$ 
and function $\wh r$ leads to an equivalent smooth structure. Actually, the two
structures are conjugated by a homeomorphism of $\OB(F,\phi)$ which preserves
$\theta$ and induces the identity on the page $F_0 := K \cup \{\theta=0\}$. As a
result, $(K,\theta)$ is an open book in $\OB(F,\phi)$, its $0$-page $F_0$ has a
(special) parametrization by $F$, and its monodromy is represented by $\phi$
(note that the vector field $\del_t$ on $\R \times F$ descends to a smooth
vector field on $\MT(F,\phi)$, so its image in $\OB(F,\phi)$ is tautologically
weakly smooth).

\begin{remark}[Smoothly Generated Monodromy Diffeomorphisms] \label{r:smooth}
Given an open book $(K,\theta)$ in $V$, one can easily find spinning vector
fields $\nu$ on $V$ that are smooth, not just weakly smooth. Thus, the monodromy
of $(K,\theta)$ has representatives which are \emph{smoothly generated}, meaning
that they can be obtained by integrating smooth spinning vector fields $\nu$ on
$V$. In particular, one can check that any representative of the monodromy which
is the identity on a neighborhood of $K$ is smoothly generated (see Lemma
\ref{l:rel} for the symplectic version of this assertion). However, Not every
representative of the monodromy is smoothly generated. The following simple
example was pointed out to me by Roussarie \cite{Ro}.

Consider in $\R^2$ a smooth vector field $\nu = 2\pi (\del_\theta + rf \del_r)$,
where $f \map \R^2 \to \R$ is a smooth function. Let $\psi, \phi \map \R \to \R$
denote the diffeomorphisms of the $x$-axis induced by the flow of $\nu$ at times
$1/2$ and $1$, respectively. Then $\phi$ and $\psi$ commute, and $\psi$ reverses
orientation while $\phi$ preserves it. These properties restrict the behavior of
$\phi$. For instance, the germ of $\phi$ at $0$ cannot have the shape $\phi(x)
= x+x^2+{}$ higher order terms. More generally, here is Roussarie's observation:
If $\phi-\id$ is not infinitely flat at $0$ then the first non-zero term in its
Taylor expansion has odd degree. Indeed, if $\phi-\id$ is not flat, it has a
fixed sign on $(0,\eps]$ for $\eps>0$ sufficiently small. Suppose that $\phi(x)
> x$ for all $x \in (0,\eps]$. Since $\psi$ is decreasing and commutes with
$\phi$, 
$$ \phi \circ \psi (x) = \psi \circ \phi (x) < \psi(x) \quad
   \text{for all $x \in (0,\eps]$.} $$
Hence, $\phi(x) < x$ for all $x \in [\psi(\eps),0)$, and this proves the claim.

In contrast, if the vector field $\nu = 2\pi (\del_\theta + rf \del_r)$ is only
assumed to be ``weakly smooth'' (namely, if $\nu$ lifts to a smooth vector field
on the blownup plane), then its return map $\phi$ on $\R_{\ge0}$ remains smooth
and can freely vary in its mapping class. Indeed, the hypothesis means that $f$
is smooth not as a function on $\R^2$ but as a function of the polar coordinates
$(r,\theta) \in \R_{\ge0} \times \S^1$. Note also that, since $f$ and $df$ are
bounded near $\{0\} \times \S^1$, the vector field $\nu$ is Lipschitz. These 
remarks equally apply to weakly smooth spinning vector fields in any dimension. 
\end{remark}

The following definition was introduced in \cite{Gi2} to establish formal links
between open books and contact structures:

\begin{definition}[Open Books and Contact Structures] \label{d:sob}
A contact structure $\xi$ on a closed manifold $V$ is \emph{supported} by an
open book $(K,\theta)$ in $V$ if it admits a Pfaff equation $\alpha$ which is
\emph{adapted to $(K,\theta)$} in the sense that:
\begin{itemize}
\item
$\alpha$ induces a positive contact form on $K$, and
\item
$d\alpha$ induces a positive symplectic form on the fibers of $\theta$.
\end{itemize}
Orientations here come from the orientation of $V$ defined by $\xi$.
\end{definition}

We will show below that an open book supporting a contact structure has some
specific geometric structure that we now describe:

\begin{definition}[Liouville Open Books] \label{d:lob}
A \emph{Liouville open book} $(K,\theta,\omega_t)$ in a closed manifold $V$ is
an open book $(K,\theta)$ whose pages $F_t := K \cup \{\theta = 2\pi t\}$ are
equipped with ideal Liouville structures $\omega_t$ $(2\pi t \in \S^1)$ having
primitives $\lambda_t$ such that: For some/any map $h \map V \to \C$ defining
$(K,\theta)$, the products $|h|\, \lambda_t$ are the restrictions to the fibers
$F_t-K$ of a global (smooth) $1$-form $\beta$ on $V$. Such a $1$-form $\beta$ is
referred to as a \emph{binding $1$-form} (associated with $h$), as it indeed
ties the forms $\omega_t$ about~$K$.

\end{definition}

In this context, we consider the affine space of weakly smooth spinning vector
fields $\nu$ on $V$ satisfying the additional condition that $\nu$ preserves the
ideal Liouille structures of pages. This means that the flow of $\nu$, which
rotates the open book $(K,\theta)$, preserves the family of forms $\omega_t$.
Equivalently, $\nu$ spans the kernel of a closed $2$-form on $V-K$ which induces
$\omega_t$ on each page $F_t$.

For such a symplectically spinning vector field $\nu$, the time~$1$ map of its
flow restricted to the ideal Liouville page $(F,\omega) := (F_0,\omega_0)$ is a
symplectic diffeomorphism $\phi$ relative to $K = \del F$. Moreover, as $\nu$
runs over its affine space, $\phi$ sweeps out a full symplectic mapping class in
$\MCG(F,\omega) := \pi_0 \DDc(F,\omega)$. This mapping class is the symplectic
monodromy of the Liouville open book.

The next lemma shows that the symplectic monodromy of a Liouville open book has
representatives which are generated by smooth symplectically spinning vector
fields and can be further assumed to be the identity on a neighborhood of $K$.
As in the usual (non-Liouville) case, however, not every representative of the
symplectic monodromy can be generated in this way.

\begin{lemma}[Binding Forms and Monodromy] \label{l:smooth}
Let $(K,\theta,\omega_t)$ be a Liouville open book in a closed manifold $V$,
and $h \map V \to \C$ a map defining $(K,\theta)$. For every binding $1$-form
$\beta$, the vector field $\nu$ on $V-K$ spanning the kernel of $d(\beta/|h|)$
and satisfying $\nu \cdot \theta = 2\pi$ extends to a smooth vector field on $V$
which is zero along $K$. Furthermore, $\beta$ can be chosen so that $\nu$ is 
$1$-periodic near $K$.  
\end{lemma}

Note that binding forms associated with any fixed defining map $h$ constitute an
affine space. Another thing to be mentioned here is that, among symplectically
spinning vector fields, those associated with binding $1$-forms generate exact
symplectic diffeomorphisms (see our comment following Proposition \ref{p:tw}).
 
\begin{proof}
First observe that $\beta \rst K$ is a contact form and defines the (common)
boundary contact structure of all ideal Liouville pages. We fix a small $\eps>0$
such that $\beta$ induces a contact form on every fiber $K_w := \{h=w\}$, $|w|
\le \eps$. We set $\alpha_w := \beta \rst{ K_w }$ and $N := \{|h| \le \eps\}$.
The hyperplane field $\tau := \Ker (\beta \rst N)$ splits as a direct sum $\tau
= \xi \oplus \xi^\bot$, where  $\xi$ is the subdistribution consisting of the 
contact structures $\xi_w := \Ker \alpha_w$, $|w| \le \eps$, and $\xi^\bot$ is
the $d\beta$-orthogonal complement of $\xi$ in $\tau$ (and determines a contact
connection over $\eps \D^2$). Now consider the following vector fields on $N$:
\begin{itemize}
\item
$\del_\alpha$ is the vector field in $\Ker dh$ whose restriction to each fiber
$K_w$ of $h$ is the Reeb field $\del_{\alpha_w}$ of $\alpha_w$.
\item
$\wt\del_\theta$ and $\wt\del_r$ are the vector fields in $\xi^\bot$ projecting
to $\del_\theta$ and $\del_r$, respectively, where $(r,\theta)$ denote polar
coordinates in $\eps \D^2$.
\end{itemize}
A routine calculation shows that the vector field $\nu$ on $V-K$ spanning the
kernel of $d(\beta/|h|)$ and satisfying $\nu \cdot \theta = 2\pi$ is given in
$N$ by
$$ \nu = 2\pi (\wt\del_\theta + a r \wt\del_r + b \del_\alpha), $$
where $r = r \circ h = |h|$ while 
$$ a := \frac{ d\beta (\wt\del_\theta, \del_\alpha) }
   { 1 + d\beta  (\del_\alpha, r \wt\del_r) } \quad \text{and} \quad
   b := \frac{ d\beta (r \wt\del_r, \wt\del_\theta) }
   { 1 + d\beta  (\del_\alpha, r \wt\del_r) }. $$
Clearly, $a$ and $b$ are smooth functions on $N$ and vanish identically along
$K$, so $\nu$ has the desired smooth extension on $V$.

We will now modify $\beta$ to obtain a binding form $\beta'$ such that the
spinning vector field $\nu'$ spanning the kernel of $d(\beta'/|h|)$ in $V-K$ is
$1$-periodic near $K$. First, we trivialize $N$ as a product $N = \eps \D^2
\times K$ so that, in the corresponding cylindrical coordinates $(r,\theta,q)$,
the vector field $\del_r$ lies in $\xi^\bot$. In other words, $\del_r$ equals
$\wt\del_r$ and, along $K = \{0\} \times K$, the $2$-plane field $\xi^\bot$ is
horizontal (namely, tangent to the disks $\D^2 \times \{q\}$, $q \in K$). It is
then an exercise to check that $\beta(\del_\theta)\, d\theta$ is a smooth form
in $N$. Now pick a function $\rho \map V \to [0,1]$ compactly supported in $N$
and equal to $1$ near $K$, and let
$$ \beta' := \beta - \rho\, \beta(\del_\theta)\, d\theta. $$
This smooth $1$-form on $V$ coincides with $\beta$ on every page, so it is a
binding form. Moreover, near $K$,
$$ \beta' = \beta - \beta(\del_\theta)\, d\theta = f\, \pi^*\alpha_0 $$
where $f$ is a positive function, $\pi$ the projection $N = \eps \D^2 \times K
\to K$ and $\alpha_0$ the restriction of $\beta$ to $K$. It follows that the
spinning vector field $\nu'$ spanning the kernel of $d(\beta'/|h|)$ is horizontal
(in the product structure of $N$) and tangent to the level sets of the function
$f/|h|$. Therefore, $\nu'$ is $1$-periodic.
\end{proof}

A practical consequence of this lemma is:

\begin{lemma}[Criterion for Smooth Generation] \label{l:rel}
Any representative of the symplectic monodromy of a Liouville open book which is
the identity near the boundary is generated by a smooth symplectically spinning
vector field.
\end{lemma}

\begin{proof}
This follows from the last assertion of Lemma \ref{l:smooth} and the fact that,
if two symplectic diffeomorphisms of an ideal Liouville domain $(F,\omega)$
coincide with the identity near $K := \del F$ and represent the same class in
$\MCG(F,\omega)$, then they are connected by a symplectic isotopy relative to a
neighborhood of $K$ (an easy way to construct such an isotopy is to use the
embeddings of Proposition \ref{p:lioufields} as in the proof of Corollary \ref
{c:exact}). 
\end{proof}

The next proposition is a variation on a wellknown construction first introduced
by Thurston--Winkelnkemper in three dimensions \cite{TW} and extended to higher
dimensions in \cite{Gi2}:

\begin{proposition}[Construction of Liouville Open Books] \label{p:tw}
Consider an ideal Liouville domain $(F,\omega)$ and a symplectic diffeomorphism
$\phi \map F \to F$ relative to $K := \del F$. The open book in $\OB(F,\phi)$ is
a Liouville open book for which the parametrization of its $0$-page by $F$ is a
symplectomorphism.
\end{proposition}

The proof below actually shows that, if the symplectic diffeomorphism $\phi$ is
the identity near $K$ and is exact (meaning that there exists an ideal Liouville
form $\lambda$ such that $\phi^*\lambda - \lambda$ is the differential of a 
function with compact support in $\Int F$), then $\phi$ is (smoothly) generated
by the spinning vector field of a binding $1$-form. 

\begin{proof}
Let $\lambda_t$ be a path of ideal Liouville forms on $(F,\omega)$ joining an
arbitrary $\lambda_0$ to $\lambda_1 := \phi^*\lambda_0$. According to Corollary
\ref{c:exact}, there is a symplectic isotopy $\psi_t$ of $F$, relative to $K$,
such that $\psi_0 = \id$ and $\psi_t^*\lambda_t - \lambda_0 = df_t$ for all $t
\in [0,1]$, where the functions $f_t$ have compact supports in $\Int F$. Then
the symplectic isotopy $\phi_t := \phi \circ \psi_t$ is relative to $K$ and
connects $\phi = \phi_0$ to a symplectic diffeomorphism $\phi_1$ which is exact
and coincides with the identity near $K$. Since $\OB(F,\phi)$ depends only on
the (smooth) mapping class of $\phi$, we assume from now on that $\phi$ is exact
and is the identity on a neighborhood of $K$. 

We now pick an ideal Liouville form $\lambda$ such that $\phi^*\lambda = \lambda
+ df_1$, where $f_1$ is a function with compact support in $\Int F$, and we
choose a path of functions $f_t$ ---~all with compact supports in $\Int F$~---
joining $f_0 := 0$ to $f_1$. Then the $1$-form $\lambda + df_t + \dot f_tdt$ on
$[0,1] \times \Int F$ is a primitive of (the pullback of) $\omega$ and descends
to a $1$-form $\wh\beta$ on $\Int \MT(F,\phi)$ (to ignore smoothing issues, take
the path $f_t$ to be constant near its endpoints).

The next step is to fix cylindrical coordinates near $K$ in $\OB(F,\phi)$ and a
map $h \map \OB(F,\phi)$ defining the obvious open book. We pick a non-negative
function $u$ on $F$, with regular level set $K = \{u=0\}$, such that:
\begin{itemize}
\item
$\dvf\lambda \cdot \log u = -1$ in  a neighborhood of $K$ (equivalently, the Lie
derivative $\dvf\lambda \cdot (u\lambda)$ is zero), and
\item
$u \circ \phi = u$ (this property is typically satisfied if $u$ is constant on
the support of $\phi$).
\end{itemize}
Then the map  
$$ (t,p) \in [0,1] \times F \mapsto u(p) e^{2i\pi t} \in \C $$ 
provides the required defining map $h \map \OB(F,\phi) \to \C$. Furthermore, the
function $u$ and the collar neighborhood of $K$ associated with $\lambda$ (cf.\ 
Proposition \ref{p:lioufields}) provide cylindrical coordinates near $K$. More
precisely, let $G := \{u \le \eps\} \subset F$ with $\eps$ small enough that
$\dvf\lambda \cdot \log u = -1$ on $G$ and $G$ is disjoint from the supports of
$\phi$ and of all functions $f_t$ $(t \in [0,1])$. Then the function $u$ and the
foliation spanned by $\dvf\lambda$ identify $G$ with $[0,\eps] \times K$. In the
same way, $N := \{|h| \le \eps\}$ is identified with $\eps \D^2 \times K$. 

It remains to see that the form $|h|\,\wh\beta$ on $\OB(F,\phi) - K$ extends to a
smooth (binding) form on $\OB(F,\phi)$. In fact, the form $u\lambda \rst{ G-K }$
is invariant under the flow of $\dvf\lambda$, so it is the pullback on $(0,\eps]
\times K$ of a $1$-form $\alpha$ on $K$. Similarly, the form $|h|\,\wh\beta \rst
{ N-K }$ is the pullback on $(\eps\D^2 - \{0\}) \times K$ of the same $1$-form
$\alpha$ on $K$. Hence, it extends smoothly across $K$.    
\end{proof}

Now the most obvious relationship between supporting and Liouville open books
is:

\begin{proposition}[Supporting Open Book are Liouville] \label{p:supliouv}
Let $(V,\xi)$ be a closed contact manifold, and $(K,\theta)$ a supporting open
book with defining map $h \map V \to \C$. Then the equations $\alpha$ of $\xi$
such that $d(\alpha/|h|)$ induces an ideal Liouville structure on each page form
a non-empty convex cone.
\end{proposition}

\begin{proof}
This follows readily from uniqueness of ideal completions of (usual) Liouville
domains (see Example \ref{x:ild2}).
\end{proof}

An equation $\alpha$ of $\xi$ as in the above proposition yields ideal Liouville
structures $\omega_t$ on the fibers of $\theta$, and $(K,\theta,\omega_t)$ is a
Liouville open book. By Lemma \ref{l:smooth}, the kernel of $d(\alpha/|h|)$ is
spanned by a smooth symplectically spinning vector field $\nu$, but it is easy 
to verify that $\nu$ is never $1$-periodic near $K$. Though it may create some
psychological disconfort, this inconvenience is not a problem. It could in fact
be remedied by replacing Liouville open books with open books whose pages are
given ``degenerate ideal Liouville structures'' (to define those objects, take
Definition \ref{d:ild} and simply substitute $u\lambda$ with $u^2\lambda$ in the
extension condition). In short, the key observation here is that, if we consider
for instance the contact form $\alpha := dz + r^2 d\theta$ in $3$-space (with
cylindrical coordinates $(r,\theta,z)$) then, away from the $z$-axis, the Reeb
field of $\alpha/r^2$ is $\del_\theta$ while the Reeb field of $\alpha/r$ is 
proportional to the vector field $\del_\theta + r^2 \del_z$.   

Proposition \ref{p:supliouv} leads to a new definition:

\begin{definition}[Liouville Open Books and Contact Structures]
Let $(K,\theta,\omega_t)$ be a Liouville open book on a closed manifold $V$, and
$h \map V \to \C$ a map defining $(K,\theta)$. A contact structure on $V$ is
(\emph{symplectically}) \emph{supported} by $(K,\theta,\omega_t)$ if it admits
a binding equation on $V$, that is, an equation $\alpha$ such that $\alpha/|h|$
induces an ideal Liouville form on each ideal Liouville page $(F_t,\omega_t)$
$(2\pi t \in \S^1)$.
\end{definition}

\begin{remark}[Uniqueness of the Binding Equation] \label{r:unique}
If it exists, the above equation $\alpha$ is unique (the defining map $h$ being
fixed). The underlying more general assertion is that, given an ideal Liouville
domain $(F,\omega)$ and an ideal Liouville form $\lambda$, the constant function
$1$ is the only function $f$ on $\Int F$ such that $d(f\lambda) = \omega$. For
$\dim F \ge 4$, the reason is purely algebraic: $f\omega$ and $\omega$ must
agree on the kernel of $\lambda$, which contains an $\omega$-symplectic space;
hence $f$ has to equal $1$. If $\dim F = 2$, non-constant solutions $f$ exist
locally, so we need a more global argument. Since
$$ d(f\lambda) = f\omega + df \wedge \lambda
 = (f + \dvf\lambda \cdot f) \omega, $$
the condition $d(f\lambda) = \omega$ reads $g + \dvf\lambda \cdot g = 0$, where 
$g := f-1$. Now any non-zero solution $g$ of this equation has to be unbounded
on every complete non-trivial orbit  of $\dvf\lambda$. The claim then follows
from $\dvf\lambda$ being complete (Proposition \ref{p:lioufields}). 
\end{remark}

If a contact structure is (symplectically) supported by a Liouville open book
then it is supported by the underlying smooth open book: To obtain an adapted 
equation in the sense of Definition \ref{d:sob}, simply replace $\alpha/|h|$ by
$\alpha/u(|h|)$ where $u \map \R_{\ge0} \to \R_{>0}$ is an increasing function
such that $u(x) = x$ for $x \ge \eps$ and $u(x) = x^2+\eps^2$ for $x \le \eps/2$
(with $\eps$ sufficiently small). 

\medskip

We now conclude this paper by showing that the inclusion of the space of contact
structures supported by a Liouville open book into the affine space of binding
forms is a weak homotopy equivalence:

\begin{proposition}[Existence and Uniqueness of Supported Contact Structures]
On a closed manifold, contact structures supported by a given Liouville open
book form a non-empty and weakly contractible subset in the space of all contact
structures. In particular, they lie in a unique isotopy class.
\end{proposition}

Note that the symplectic orientation of the pages, together with their natural
coorientation, determines an orientation of the ambient manifold. It is implicit
in this statement that the supported contact structures are positive for this
orientation. 

\begin{proof}
Let $V$ be the ambient closed $2n+1$-manifold, $(K,\theta,\omega_t)$ a Liouville
open book in $V$ and $h \map V \to \C$ a map defining $(K,\theta)$. For any
binding form $\beta$ on $V$ (associated with $h$), we can find an $\eps>0$ such
that $\beta$ induces a contact form on every fiber $K_w := \{h=w\}$ with $|w|
\le \eps$. We fix a non-decreasing function $f \map \R_{\ge0} \to \R$ such that
$f(x) = x$ for $x \le \eps/2$ and $f(x) = 1$ for $x \ge \eps$. Then, for $c \ge
0$, we define
$$ \beta_c := \beta + c\, |h|\, f(|h|)\, d\theta. $$
Clearly, $\beta_c/|h|$ coincides with $\beta/|h|$ on every page. Therefore, if
$\beta_c$ is a contact form, the contact structure it defines is symplectically
supported by our Liouville open book. We claim that $\beta_c$ is a contact form
for all sufficiently large $c$, and in fact for all $c \ge 0$ if $\beta$ itself
is already a contact form. To see this, we set $r := |h|$ and we write 
\begin{multline*}
   \beta_c \wedge (d\beta_c)^n
 = nc rf'(r)\, dr \wedge d\theta \wedge \beta \wedge (d\beta)^{n-1} \\
 + cf(r)\, d\theta \wedge (r\, d\beta + n \beta \wedge dr) \wedge (d\beta)^{n-1}
 + \beta \wedge (d\beta)^n.
\end{multline*} 
Since $\beta$ induces a contact form on each fiber $K_w$ of $h$ with $|w| \le
\eps$, the term $rf'(r)\, dr \wedge d\theta \wedge \beta \wedge (d\beta)^{n-1}$
is a positive volume form provided $f'(r) \ne 0$. On the other hand, for all
$r>0$,
$$ f(r)\, d\theta \wedge (r\, d\beta + n \beta \wedge dr) \wedge (d\beta)^{n-1}
 = r^{n+1} f(r)\, d\theta \wedge \bigl(d(\beta/r)\bigr)^n $$ 
is also a positive volume form. The claim follows.

Now consider a $k$-sphere $\xi_s$, $s \in \S^k$, of contact structures supported
by the Liouville open book $(K,\theta,\omega_t)$. By Remark \ref{r:unique},
every contact structure $\xi_s$ has a unique binding equation $\alpha_s$, and
the forms $\alpha_s$, $s \in \S^k$, depend continuously on $s$. Since binding
forms (associated with a fixed $h$) constitute an affine space, we can find a
$(k+1)$-disk of binding forms $\beta_s$, $s \in \D^{k+1}$, such that $\beta_s =
\alpha_s$ for all $s \in \S^k = \del\D^{k+1}$. We choose $\eps>0$ small enough
that each $\beta_s$, $s \in \D^{k+1}$, induces a contact form on all fibers
$K_w$ with $|w| \le \eps$, and we apply our claim twice:
\begin{itemize}
\item
For some $c_0$ sufficiently large, the forms
$$ \beta_{s,c_0} := \beta_s + c_0 |h|\, f(|h|)\, d\theta, \quad
   s \in \D^{k+1}, $$
constitute a $(k+1)$-disk of contact forms.
\item
For the same value $c_0$, the forms 
$$ \alpha_{s,c} := \alpha_s + c\, |h|\, f(|h|)\, d\theta, \quad
   s \in \S^k, \ c \in [0,c_0], $$
constitute a homotopy of $k$-spheres of contact forms between the original
$k$-sphere $\alpha_s = \alpha_{s,0}$, $s \in \S^k$, and the $k$-sphere
$\alpha_{s,c_0} = \beta_{s,c_0}$, $s \in \S^k$, which bounds a $(k+1)$-disk of
contact forms.
\end{itemize}
Since all these contact forms are binding forms, all the contact structures they
define are supported by the Liouville open book $(K,\theta,\omega_t)$, and so
our argument shows that the $k$-sphere $\xi_s$, $s \in \S^k$, is nulhomotopic in
the space of contact structures supported by $(K,\theta,\omega_t)$.
\end{proof}


\begin{thebibliography}{MMM}
\frenchspacing

\bibitem[BEE]{BEE}
F.~\textsc{Bourgeois}, T.~\textsc{Ekholm} and Y.\textsc{Eliashberg} ---
Effect of Legendrian surgery. \
\textit{Geom. Topol.} \textbf{16} (2012), 301--389.

\bibitem[CE]{CE}
K.~\textsc{Cieliebak} and Y.~\textsc{Eliashberg} ---
\textit{From Stein to Weinstein and Back}. \  
Amer. Math. Soc. Colloq. Publ. \textbf{59}, Amer. Math. Soc. 2012.

\bibitem[CFH]{CFH}
K.~\textsc{Cieliebak}, A.~\textsc{Floer} and H.~\textsc{Hofer} ---
Symplectic homology, II: A general construction. \ 
\textit{Math. Z.} \textbf{218} (1995), 103--122. 

\bibitem[Co1]{Co1}
S.~\textsc{Courte} ---
Contact manifolds with symplectomorphic symplectizations. \ 
\textit{Geom. Topol.} \textbf{18} (2014), 1--15.

\bibitem[Co2]{Co2}
S.~\textsc{Courte} ---
\textit{h-Cobordismes en G\'eom\'etrie Symplectique}. \ 
Th\`ese de Doctorat, \'Ecole Normale Sup\'erieure de Lyon, juin 2015.

\bibitem[Do]{Do}
S.~\textsc{Donaldson} ---
Symplectic submanifolds and almost-complex geometry. \ 
\textit{J. Diff. Geom.} \textbf{44} (1996), 666--705. 

\bibitem[EG]{EG}
Y.~\textsc{Eliashberg} and M.~\textsc{Gromov} --- 
Convex symplectic manifolds. \ 
\textit{Several Complex Variables and Complex Geometry}, part~2, 135--162,
Proc. Sympos. Pure Math. \textbf{52}, Amer. Math. Soc. 1991.

\bibitem[EKP]{EKP}
Y.~\textsc{Eliashberg}, S.~S.~\textsc{Kim} and L.~\textsc{Polterovich} ---
Geometry of contact transformations and domains: Orderability versus squeezing.  \ 
\textit{Geom. Topol.} \textbf{10} (2006), 1635--1747.

\bibitem[Ge]{Ge}
H.~\textsc{Geiges} ---
Symplectic manifolds with disconnected boundary of contact type. \ 
\textit{Internat. Math. Res. Notices} (1994), 23--30.

\bibitem[Gi1]{Gi1}
E.~\textsc{Giroux} ---
Convexit\'e en topologie de contact. \ 
\textit{Comment. Math. Helv.} \textbf{66} (1991), 637--677.

\bibitem[Gi2]{Gi2}
E.~\textsc{Giroux} ---
G\'eom\'etrie de contact~: de la dimension trois vers les dimensions sup\'erieures. \ 
\textit{Proceedings of the International Congress of Mathematicians}, vol.~2,
405--414, Higher Ed. Press 2002.

\bibitem[La]{La}
F.~\textsc{Laudenbach} ---
Engouffrement symplectique et intersections lagrangiennes. \ 
\textit{Comment. Math. Helv.} \textbf{70} (1995), 558--614.

\bibitem[MNW]{MNW}
P.~\textsc{Massot}, K.~\textsc{Niederkr\"uger} and C.~\textsc{Wendl} ---
Weak and strong fillability of higher dimensional contact manifolds. \ 
\textit{Invent. Math.} \textbf{192} (2013), 287--373.

\bibitem[Mc]{Mc}
D.~\textsc{McDuff} ---
Symplectic manifolds with contact type boundaries. \ 
\textit{Invent. Math.} \textbf{103} (1991), 651--671.

\bibitem[Ro]{Ro}
R.~\textsc{Roussarie} ---
Private communication by e-mail, July 2017.

\bibitem[Se]{Se}
P.~\textsc{Seidel} ---
A biased view of symplectic cohomology. \ 
\textit{Current Developments in Mathematics}, 211--253, Int. Press 2008.

\bibitem[TW]{TW}
W.~\textsc{Thurston} and H.~\textsc{Winkelnkemper} ---
On the existence of contact forms. \ 
\textit{Proc. Amer. Math. Soc.} \textbf{52} (1975), 345--347. 

\bibitem[Vi]{Vi}
C.~\textsc{Viterbo} ---
Functors and computations in Floer homology with applications, I. \ 
\textit{Geom. Funct. Anal.} \textbf{9} (1999), 985--1033.

\bibitem[We]{We}
A.~\textsc{Weinstein} ---
On the hypotheses of Rabinowitz' periodic orbit theorems. \ 
\textit{J. Differential Equations} \textbf{33} (1979), 353--358.

\end{thebibliography}
\end{document}